\documentclass[12pt,reqno]{amsart}

\usepackage[latin1]{inputenc}
\usepackage{amsmath}
\usepackage{amsfonts}
\usepackage{amssymb}
\usepackage{graphics}
\usepackage{enumerate}
\usepackage{enumitem}
\usepackage{subfigure}
\usepackage{amssymb,amsmath,amsthm,amscd,epsf,latexsym,verbatim,graphicx,amsfonts}
\input epsf.tex

\usepackage{amscd}
\usepackage{amsmath}
\usepackage{amssymb}
\usepackage{amsthm}
\usepackage{epsf}
\usepackage{latexsym}
\usepackage{verbatim}

\theoremstyle{plain}

\newtheorem{theorem}{Theorem}[section]

\newtheorem{corollary}[theorem]{Corollary}
\newtheorem{lemma}[theorem]{Lemma}
\newtheorem{prop}[theorem]{Proposition}
\theoremstyle{definition}

\theoremstyle{remark}

\newcommand{\bbZ}{\mathbb{Z}}
\newcommand{\bbR}{\mathbb{R}}

\newcommand{\bbD}{\mathbb{D}}
\newcommand{\bbN}{\mathbb{N}}
\newcommand{\bbP}{\mathbb{P}}

\newcommand{\mca}{\mathcal{A}}
\newcommand{\mcb}{\mathcal{B}}

\newcommand{\mcf}{\mathcal{F}}

\newcommand{\mcz}{\mathcal{Z}}

\newcommand{\mcn}{\mathcal{N}}
\newcommand{\mcp}{\mathcal{P}}

\newcommand{\mmu}{(M_{\mu})}

\newcommand{\nri}{n\rightarrow\infty}
\newcommand{\jri}{j\rightarrow\infty}
\newcommand{\kri}{k\rightarrow\infty}

\DeclareMathOperator*{\supp}{supp}

\begin{document}
\title[ ] {Relative Asymptotics for General Orthogonal Polynomials}

\bibliographystyle{plain}

\thanks{  }

\maketitle

\begin{center}
\textbf{Brian Simanek}\\
\textit{\small{Baylor University Department of Mathematics\\
One Bear Place $\#$97328, Waco, TX 76798}}\\
\small{Brian$\_$Simanek@baylor.edu}
\end{center}

\begin{abstract}
In this paper, we study right limits of the Bergman Shift matrix.  Our results have applications to ratio asymptotics, weak asymptotic measures, relative asymptotics, and zero counting measures of orthogonal and orthonormal polynomials.  Of particular interest are the applications to random orthogonal polynomials on the unit circle and real line.
\end{abstract}

\section{Introduction}\label{intro}

Let $\mu$ be a positive probability measure with compact and infinite support in the complex plane.  Given such a measure, it is well-known how one forms the sequence of orthonormal polynomials for this measure, which we denote by $\{p_n(z;\mu)\}_{n\geq0}$ and satisfy
\[
\int p_n(z;\mu)\overline{p_m(z;\mu)}d\mu(z)=\delta_{m,n}.
\]
The polynomial $p_n(z;\mu)$ has positive leading coefficient, which we denote by $\kappa_n(\mu)$.  The monic polynomial $\kappa_n(\mu)^{-1}p_n(z;\mu)$ will be denoted by $P_n(z;\mu)$ and is of interest in its own right.

It is both interesting and informative to try to understand the relationship between the measure $\mu$ and the corresponding sequence of orthonormal polynomials.  There are many ways that one can study the orthogonal polynomial asymptotics.  One of the most interesting kinds of asymptotic behavior is known as \textit{ratio asymptotics}, and concerns the following limits:
\begin{align}
\label{rd2}\lim_{\nri}\frac{p_{n-1}(z;\mu)}{p_{n}(z;\mu)}\\
\label{ratiodef}\lim_{\nri}\frac{P_{n-1}(z;\mu)}{P_{n}(z;\mu)}
\end{align}
if these limits exist.  When studying ratio asymptotics, one is concerned with the existence of the limit, the domain of the limit, and a functional form of the limit.  

A second kind of asymptotic behavior is known as \textit{weak asymptotic measures}, and concerns the weak limits of the measures $\{|p_n(z;\mu)|^2d\mu(z)\}_{n\in\bbN}$ as $\nri$.  Since $\mu$ has compact support, weak limits always exist and one can gain insight into properties of the measure $\mu$ by understanding these weak limits.  In some cases, it is more practical to study the weak asymptotic moments, that is, to understand the moments of all the weak asymptotic measures.  In general, the moments of a measure do not determine the measure, so this is a less precise object to study, but we will see that we can make precise statements about these moments even when we cannot identify the measure.  If each of the following limits exists:
\[
\lim_{\nri}\int z^j|p_n(z;\mu)|^2d\mu(z),\qquad j\in\bbN,
\]
then we will say that the measure $\mu$ exhibits \textit{weak asymptotic moments}.

A third way to study orthogonal polynomial asymptotics is to study the asymptotic distribution of the zeros of $P_n$ as $n$ becomes large.  Let $\sigma_n(\mu)$ be the probability measure that assigns weight $n^{-1}$ to each zero of $P_n$, where we count each zero a number of times equal to its multiplicity as a zero.  If each of the following limits exists:
\[
\lim_{\nri}\int z^jd\sigma_n(z),\qquad j\in\bbN,
\]
then we will say that the measure $\mu$ exhibits \textit{Ces\`{a}ro weak asymptotic moments}.

In order to understand precisely when each of these asymptotic behaviors can be observed, we need to consider the Bergman Shift operator.  For a compactly supported measure $\mu$, the Bergman Shift is a bounded operator on $L^2(\mu)$ that maps $f(z)$ to $zf(z)$.  If we let $\mcp(\mu)$ denote the closure of the polynomials in $L^2(\mu)$, then it is clear that the Bergman Shift maps $\mcp(\mu)$ to itself.  If we use the orthonormal polynomials for the measure $\mu$ as a basis for $\mcp(\mu)$, then the orthogonality condition implies that the resulting matrix representation of the Bergman Shift is a Hessenberg matrix:
\[
M=\begin{pmatrix}
M_{11} & M_{12} & M_{13} & M_{14} & \cdots\\
M_{21} & M_{22} & M_{23} & M_{24} & \cdots\\
0 & M_{32} & M_{33} & M_{34} & \cdots\\
0 & 0 & M_{43} & M_{44} & \cdots\\
0 & 0 & 0 & M_{54} & \cdots\\
\vdots & \vdots & \vdots & \vdots & \ddots
\end{pmatrix}
\]
In some cases, we will use the notation $M_{\mu}$ to emphasize the measure that gives rise to this matrix.  If the support of $\mu$ is a compact subset of the real line, then the matrix $M$ is often called a Jacobi matrix.  In this case, the matrix $M_{\mu}$ is symmetric, its entries are zero away from the three main diagonals, and the matrix entries can be given explicitly in terms of the coefficients appearing in the three term recurrence relation satisfied by the corresponding orthonormal polynomials (see \cite{Rice}).  If the support of $\mu$ is a subset of the unit circle, then $M$ is sometimes called the GGT matrix for the measure $\mu$ (see \cite[Section 4.1]{OPUC1}) and its entries can be described explicitly in terms of the recursion coefficients for the corresponding monic orthogonal polynomials.

The structure of the matrix $M$ for general $\mu$ is more difficult to describe.  Some interesting properties appear in \cite{Shift,RatioHessI,RatioHessII}, all of which focus on the relationship between properties of the measure and the entries of $M$ as we move along one of its diagonals.  We will continue this trend of using the limiting behavior of the entries of $M$ along its diagonals to deduce properties of the measure $\mu$ and the corresponding orthonormal polynomials.

In \cite{RatioHessI}, the author provides a precise description of when the monic orthogonal polynomials and orthonormal polynomials for $\mu$ exhibit ratio asymptotics.  The result is stated as a condition on the entries of the Bergman Shift matrix $M_{\mu}$.  Our first main goal in the present work is to prove a relative version of that result, so we must introduce several notions of relative asymptotics.

For any compactly supported measure $\mu$, let us define $R_{\mu}:=\sup\{|z|:z\in\supp(\mu)\}$.  We will say that two measures $\mu$ and $\nu$ exhibit \textit{relative ratio asymptotics} if there is a non-negative integer $q$ so that
\begin{align}\label{relratdef}
\lim_{\nri}\left(\frac{P_{n-1}(z;\mu)}{P_{n}(z;\mu)}-\frac{P_{n-q-1}(z;\nu)}{P_{n-q}(z;\nu)}\right)=0,\qquad |z|>\max\{R_{\mu},R_{\nu}\}.
\end{align}
We will say that two measures $\mu$ and $\nu$ exhibit \textit{relative weak asymptotic moments} if there is a non-negative integer $q$ so that
\begin{align}\label{jdiff}
\lim_{\nri}\left(\int z^j|p_n(z;\mu)|^2d\mu(z)-\int z^j|p_{n-q}(z;\nu)|^2d\nu(z)\right)=0,\qquad j\in\bbN.
\end{align}
We will say that two measures $\mu$ and $\nu$ exhibit \textit{relative Ces\`{a}ro weak asymptotic moments} if
\begin{align}\label{jdiff2}
\lim_{\nri}\left(\int z^jd\sigma_n(\mu)(z)-\int z^jd\sigma_n(\nu)(z)\right)=0,\qquad j\in\bbN.
\end{align}
A different, but related notion of relative asymptotics is discussed in \cite{MNT2}.

Results on the existence of relative ratio asymptotics can be found in \cite[Theorem 3.7]{Beck}, \cite[Theorem 6]{RobSant}, and \cite[Theorem 1.5]{GolZlat}.  While our main result on relative ratio asymptotics is similar in both content and proof to \cite[Theorem 3.7]{Beck}, we make the important generalization that one can take $n$ to infinity along subsequences (which was observed in \cite{RobSant}) and even use different subsequences for each measure.

For a given measure $\mu$, and any $j,n\in\bbN_0:=\bbN\cup\{0\}$, let us define $h_{j,n}(\mu)$ by
\[
h_{j,n}(\mu):=\frac{\kappa_{n-1-j}(\mu)}{\kappa_{n-1}(\mu)}(M_{\mu})_{n-j,n}.
\]
Now we can state our main result concerning relative ratio asymptotics.

\begin{theorem}\label{iffrelrat}
Let $\mu$ and $\nu$ be two compactly supported and finite measures and let $q\in\bbN_0$ be fixed.  We have
\begin{align}\label{shiftseq}
\lim_{\nri}\left(\frac{P_{n-1}(z;\mu)}{P_{n}(z;\mu)}-\frac{P_{n-q-1}(z;\nu)}{P_{n-q}(z;\nu)}\right)=0
\end{align}
if and only if for every $j\in\bbN_0$ it holds that
\begin{align}\label{hcond}
\lim_{\nri}\left(h_{j,n}(\mu)-h_{j,n-q}(\nu)\right)=0.
\end{align}
\end{theorem}

One might be able to obtain Theorem \ref{iffrelrat} as a direct consequence of \cite[Theorem 3.7]{Beck} or \cite[Theorem 6]{RobSant}, at least in the case $q=0$.  However, the setting we consider is slightly different than in both of those papers, so we will include a proof in Section \ref{prf} for completeness.

It is also meaningful to discuss a relative version of the relation (\ref{rd2}).  Indeed, if  $\{n_k\}_{k=1}^{\infty}=\mcn(\mu)\subseteq\bbN$ and $\{m_k\}_{k=1}^{\infty}=\mcn(\nu)\subseteq\bbN$ are two subsequences, then we will say that two measures $\mu$ and $\nu$ exhibit \textit{normalized relative ratio asymptotics through $\mcn(\mu)$ and $\mcn(\nu)$} precisely when
\begin{align}\label{nrel}
\lim_{\kri}\left(\frac{p_{n_k+j-1}(z;\mu)}{p_{n_k+j}(z;\mu)}-\frac{p_{m_k+j-1}(z;\nu)}{p_{m_k+j}(z;\nu)}\right)=0,\qquad |z|>\max\{R_{\mu},R_{\nu}\},\qquad j\in\bbZ.
\end{align}
If $\mcn(\mu)=\mcn(\nu)=\bbN$, then we will say $\mu$ and $\nu$ exhibit \textit{normalized relative ratio asymptotics}.  A normal families argument implies that the convergence in (\ref{nrel}) is always uniform on compact subsets of $\{z:|z|>\max\{R_{\mu},R_{\nu}\}\}$.

Existence of the limit (\ref{nrel}) is a stronger conclusion than the existence of the limit (\ref{relratdef}) because it assumes some convergence of the ratio of consecutive leading coefficients.  Thus, normalized relative ratio asymptotics will be equivalent to a stronger statement than (\ref{hcond}).  To state this condition, we must recall the definition of a right limit of a matrix (see \cite{RatioHessII}).

For a matrix $A$ whose rows and columns are indexed by the natural numbers, let $A_n^{(m)}$ be the $2m+1\times2m+1$ sub-matrix of $A$ centered at $A_{n,n}$.  We will say that a matrix $X$ is a right limit of the matrix $A$ if there is a subsequence $\mcn\subseteq\bbN$ so that for every $m\in\bbN$ it holds that
\[
\lim_{{\nri}\atop{n\in\mcn}}A_n^{(m)}=X_0^{(m)}.
\]
In this case, we will also say that $A$ approaches $X$ through $\mcn$.  If the matrix elements of $A$ are bounded (which they are in the cases we are considering), then one can always find right limits.  Also, notice that right limits are bi-infinite, even though the matrix $A$ is not.  In the context of orthogonal polynomials on the real line, right limits have been used extensively for asymptotic calculations (see \cite{Rice,VA95} and references therein).  For more on right limits, we refer the reader to \cite{RatioHessII} and \cite[Chapter 7]{Rice}.

The recent publications \cite{RatioHessI,RatioHessII} shed some light on the relationship between right limits of the matrix $M_{\mu}$ and ratio asymptotics for general orthonormal polynomials under the added condition that
\begin{align}\label{kapo}
\liminf_{\nri}\frac{\kappa_n}{\kappa_{n+1}}>0.
\end{align}
Note that this is a very mild condition that is satisfied by many measures of interest including area measure on a smooth Jordan region \cite{Shift} and measures on the unit circle with non-vanishing absolutely continuous component \cite{OPUC1}.  Using essentially the same techniques as \cite{RatioHessI,RatioHessII}, we will show the following result (recall the definition of $R_{\mu}$ given before Equation (\ref{relratdef})):

\begin{theorem}\label{iffrlim}
Let $\mu$ be a finite measure with compact and infinite support so that (\ref{kapo}) is satisfied.  If $\mcn\subseteq\bbN$ is a subsequence, then the following statements are equivalent:
\begin{itemize}
\item[I)]  There is a sequence of functions $\{f_j\}_{j\in\bbZ}$, each of which is analytic in $\{z:R_{\mu}<|z|\leq\infty\}$ so that for all $j\in\bbZ$ it is true that
\begin{align}\label{fjdef}
\lim_{{\nri}\atop{n\in\mcn}}\frac{p_{n+j-1}(z;\mu)}{p_{n+j}(z;\mu)}=f_j(z),\qquad |z|>R_{\mu},
\end{align}
where the convergence is uniform on the exterior of $\{z:|z|>r\}$ for every $r>R_{\mu}$.
\item[II)]  The matrix $M_{\mu}$ approaches a right limit as $\nri$ through $\mcn$.
\end{itemize}
The implication (II) implies (I) is valid even if condition (\ref{kapo}) is not satisfied.
\end{theorem}

\noindent\textit{Remark.}  The proof of Theorem \ref{iffrlim} will make it clear that if the condition (II) of Theorem \ref{iffrlim} applies to a measure $\mu$, then the right limit of $M_{\mu}$ determines the sequence $\{f_j\}_{j\in\bbZ}$.  If condition (\ref{kapo}) is satisfied, then the sequence $\{f_j\}_{j\in\bbZ}$ in condition (I) determines the right limit also.

\vspace{2mm}

Theorem \ref{iffrlim} and the above remark immediately provide the following characterization of pairs of measures that exhibit normalized relative ratio asymptotics.

\begin{corollary}\label{genrelrat}
Let $\mu$ and $\nu$ be two positive and finite measures, each having compact and infinite support and satisfying (\ref{kapo}).  Furthermore, let $\{n_k\}_{k=1}^{\infty}=\mcn(\mu)\subseteq\bbN$ and $\{m_k\}_{k=1}^{\infty}=\mcn(\nu)\subseteq\bbN$ be two subsequences.  The following statements are equivalent:
\begin{itemize}
\item[i)]  The measures $\mu$ and $\nu$ exhibit normalized relative ratio asymptotics through $\mcn(\mu)$ and $\mcn(\nu)$.
\item[ii)]  $M_{\mu}$ approaches a right limit $X$ through $\{n_k\}_{k\in\mcn^*\subseteq\bbN}$ if and only if $M_{\nu}$ approaches $X$ through $\{m_k\}_{k\in\mcn^*}$.
\end{itemize}
In particular, the measures $\mu$ and $\nu$ exhibit normalized relative ratio asymptotics if and only if the unique right limit of $M_{\mu}-M_{\nu}$ is $0$.  The implication (ii) implies (i) is valid even if the condition (\ref{kapo}) is not satisfied.
\end{corollary}

\noindent\textit{Remark.}  We will see by example in Section \ref{examp} that the hypothesis (\ref{kapo}) is essential for the implication (i) implies (ii) in the theorem.

\vspace{2mm}

Given a compactly supported finite measure with infinitely many points in its support, let $\mca_{\mu}$ be the collection of analytic functions that occur as uniform limits of the sequence $\{p_{n-1}(z;\mu)/p_{n}(z;\mu)\}_{n\in\bbN}$ as $\nri$ through a subsequence (\cite[Proposition 2.2]{RatioHessI} implies that $\mca_{\mu}$ is not empty).  Let us also denote by $\mcb_{\mu}$ the collection of right limits of $M_{\mu}$, and we have already remarked that $\mcb_{\mu}$ is not empty.  Theorem \ref{iffrlim} easily implies the following corollary.

\begin{corollary}\label{samerat}
Let $\mu$ and $\nu$ be two compactly supported and finite measures each having infinite support.  If $\mcb_{\mu}=\mcb_{\nu}$, then $\mca_{\mu}=\mca_{\nu}$.
\end{corollary}

In the next section, we will discuss some elementary consequences of the results in this section.  In particular, we will be able to prove that relative ratio asymptotics implies relative weak asymptotic moments.  Section \ref{random} discusses various universal properties of random orthogonal polynomials on the unit circle and real line by showing that such polynomials exhibit behavior that is dense in the appropriate spaces.  The proofs of the above results are contained in Section \ref{prf}, while Section \ref{examp} contains some additional examples.

\vspace{4mm}

\noindent\textbf{Acknowledgements}  We thank Nikos Stylianopoulos for bringing the important paper \cite{Beck} to our attention.  We thank the anonymous referees for useful suggestions and especially for their suggestion to consult references \cite{GHVA,GNVA,MNT,ND79,NVA92,VA91,VA95}.  We also thank P. Nevai for useful discussion of previous results and for helping us find the reference \cite{MNT}.

\section{Asymptotic Relationships}\label{asympr}

The relationship between weak asymptotic moments, and Ces\`{a}ro weak asymptotic moments was proven in the context of orthogonal polynomials on the real line in \cite[Theorem 2]{MNT}.  This was generalized in several ways by the author in \cite{RatioHessI}, the results of which are summarized in the following theorem:

\begin{theorem}\label{past}
Let $\mu$ be a finite measure with compact and infinite support.  

\vspace{2mm}

(i) If $\mu$ exhibits weak asymptotic moments, then it exhibits Ces\`{a}ro weak asymptotic moments.

\vspace{2mm}

(ii) If the monic orthogonal polynomials for $\mu$ exhibit ratio asymptotics, then $\mu$ exhibits weak asymptotic moments.
\end{theorem}

The results in Section \ref{intro} allow us to prove the following relative version of that result:

\begin{theorem}\label{present}
Let $\mu$ and $\nu$ be two finite measures with compact and infinite support.  

\vspace{2mm}

(i) If $\mu$ and $\nu$ exhibit relative weak asymptotic moments, then they also exhibit relative Ces\`{a}ro weak asymptotic moments.

\vspace{2mm}

(ii) If $\mu$ and $\nu$ exhibit relative ratio asymptotics, then they also exhibit relative weak asymptotic moments.
\end{theorem}

To prove Theorem \ref{present}, we will require the following lemma, the proof of which is contained in the proof of \cite[Theorem 1.2]{RatioHessI}:

\begin{lemma}\label{hprod}
Let $\pi_n$ be the projection onto the span of $\{1,\ldots,z^{n-1}\}$ inside $\mcp(\mu)$.
\begin{enumerate}[before=\itshape,font=\normalfont]
\item For any natural number $m\geq1$, one can write
\[
\left((\pi_{n}M_{\mu}\pi_{n})^m\right)_{n,n}=h_{m-1,n}(\mu)+\beta_{n,m},
\]
where $\beta_{n,m}$ can be expressed as a sum of products of factors of the form $h_{j,n-k}(\mu)$ for appropriate values of $j,k\in\{0,1,\ldots,m-2\}$.

\item  For any natural number $m\geq1$,  $(M_{\mu}^m)_{n,n}$ can be expressed as a sum of products of factors of the form $h_{j,n-k}(\mu)$ where $j\in\{0,1,\ldots,m-1\}$ and $k\in\{-m+1,-m+2,\ldots,m-2,m-1\}$
\end{enumerate}
\end{lemma}

\noindent\textit{Remark.}  Lemma \ref{hprod}(a) is stated in the proof of \cite[Theorem 3.7]{Beck}.  A result similar to Lemma \ref{hprod}(b) for tri-diagonal matrices is used in the proof of \cite[Lemma 3]{ND79} (see also \cite[Equation 4.3]{GHVA} for a related result for banded matrices).

\vspace{2mm}

\begin{proof}[Proof of Theorem \ref{present}]
(i)  We proceed exactly as in the proof of Theorem \ref{past}(i) in \cite{RatioHessI}.  Indeed, suppose $\mu$ and $\nu$ exhibit relative weak asymptotic moments.  We know from \cite[Proposition 2.3]{WeakCD} that
\begin{align*}
&\left|\int z^jd\sigma_n(\mu)(z)-\int z^jd\sigma_n(\mu)(z)-\frac{1}{n}\sum_{k=0}^{n-1}\left(\int z^j|p_n(z;\mu)|^2d\mu(z)-\int z^j|p_n(z;\nu)|^2d\nu(z)\right)\right|\\
&\qquad\qquad\qquad\qquad\leq\frac{2j(R_{\mu}+R_{\nu})}{n},\qquad\qquad\qquad j\in\bbN.
\end{align*}
As $\nri$, the sum over $k$ tends to zero by hypothesis and the right hand side tends to zero, so the difference between the first two integrals must also tend to zero.

\vspace{2mm}

(ii):  Suppose that $\mu$ and $\nu$ exhibit relative ratio asymptotics.  Theorem \ref{iffrelrat} implies that for some $q\in\bbN_0$
\[
\lim_{\nri}\left(h_{j,n}(\mu)-h_{j,n-q}(\nu)\right)=0,\qquad j\in\bbN_0.
\]
Combining this with the second part of Lemma \ref{hprod} and noting that
\[
\int z^m|p_n(z;\mu)|^2d\mu(z)=\left(M_{\mu}^m\right)_{n+1,n+1}
\]
proves the desired result. 
\end{proof}

\vspace{4mm}

In the next section, we will prove the results stated in Section \ref{intro}.

\section{Proof of the Main Theorems}\label{prf}

Our main objective in this section is to prove the results stated in Section \ref{intro}.  Our methods will be similar to those used in \cite{RatioHessI}; indeed we will rely on results from \cite{RatioHessI} in some of our proofs.  We begin by recalling the following fact, which is contained in Proposition 2.2 in \cite{WeakCD}.

\begin{prop}\label{det}
Let $\pi_n$ be the projection onto the $n$-dimensional subspace given by the span of $\{1,z,\ldots,z^{n-1}\}$ inside $\mcp(\mu)$.  The polynomial $P_n(z;\mu)$ and the matrix $M_{\mu}$ are related by
\begin{align}\label{keydet}
P_n(z;\mu)=\det\left(z-\pi_n M_{\mu}\pi_n\right),
\end{align}
where the determinant is the usual definition of the determinant of a finite matrix.
\end{prop}

We will also make repeated use of the fact that $\{P_{n-1}(z;\mu)/P_n(z;\mu)\}_{n\in\bbN}$ (and hence also $\{p_{n-1}(z;\mu)/p_n(z;\mu)\}_{n\in\bbN}$) is a normal family (see \cite[Proposition 2.2]{RatioHessI}).

Our first task is to prove Theorem \ref{iffrelrat}, which we will do by appealing to Lemma \ref{hprod}.

\begin{proof}[Proof of Theorem \ref{iffrelrat}:]
Suppose (\ref{hcond}) holds for all $j\in\bbN_0$.  Let $\rho_{M_{\mu},n}(z)=(z-\pi_nM_{\mu}\pi_n)^{-1}$ be the resolvent of the truncated $M_{\mu}$ matrix.  By (\ref{keydet}) and Cramer's rule, we can write (when $|z|$ is sufficiently large)
\begin{align}\label{cram1}
\frac{P_{n-1}(z;\mu)}{P_{n}(z;\mu)}&=\rho_{M_{\mu},n}(z)_{n,n}=\sum_{m=0}^{\infty}\frac{\left((\pi_{n}M_{\mu}\pi_{n})^m\right)_{n,n}}{z^{m+1}}
\end{align}
(see also \cite[equation (2.21)]{SimonWeak}).  Therefore, when $|z|$ is sufficiently large, we have
\begin{align}\label{cram11}
\frac{P_{n-1}(z;\mu)}{P_{n}(z;\mu)}-\frac{P_{n-q-1}(z;\nu)}{P_{n-q}(z;\nu)}&=\sum_{m=0}^{\infty}\frac{\left((\pi_{n}M_{\mu}\pi_{n})^m\right)_{n,n}-\left((\pi_{n-q}M_{\nu}\pi_{n-q})^m\right)_{n-q,n-q}}{z^{m+1}}.
\end{align}
We will show that all of the Laurent coefficients of this expansion converge to $0$ as $\nri$.  This and \cite[Proposition 2.2]{RatioHessI} will prove the desired result.

To this end, recall that we are assuming that $h_{j,n}(\mu)=h_{j,n-q}(\nu)+o(1)$ as $\nri$.  Therefore, Lemma \ref{hprod} implies that for each $m\in\bbN_0$, $\left((\pi_{n}M_{\mu}\pi_{n})^m\right)_{n,n}=\left((\pi_{n-q}M_{\nu}\pi_{n-q})^m\right)_{n-q,n-q}+o(1)$ as $\kri$, which is our desired conclusion.

For the converse, suppose the monic orthogonal polynomials for $\mu$ and $\nu$ satisfy the relation (\ref{shiftseq}).  In this case, we know that each Laurent coefficient in (\ref{cram11}) converges to $0$ as $\nri$.  We will proceed by induction to show that (\ref{hcond}) holds.

For the base case of the induction, we consider the $m=1$ term in (\ref{cram11}).  The numerator is $h_{0,n}(\mu)-h_{0,n-q}(\nu)$, so we have proven (\ref{hcond}) for $j=0$.  This will serve as the base case of our induction.  As our induction hypothesis, assume that (\ref{hcond}) holds for all $j<k$ for some fixed $k\in\bbN$.  If we look at the $m=k+1$ term in (\ref{cram11}), then we see that the coefficient is $\left((\pi_{n}M_{\mu}\pi_{n})^{k+1}\right)_{n,n}-\left((\pi_{n-q}M_{\nu}\pi_{n-q})^{k+1}\right)_{n-q,n-q}$, which therefore must tend to zero as $\nri$.  Lemma \ref{hprod} and the induction hypothesis imply that
\[
\lim_{\nri}\left(\left((\pi_{n}M_{\mu}\pi_{n})^{k+1}\right)_{n,n}-h_{k,n}(\mu)-\left(\left((\pi_{n-q}M_{\nu}\pi_{n-q})^{k+1}\right)_{n-q,n-q}-h_{k,n-q}(\nu)\right)\right)=0.
\]
This implies (\ref{hcond}) holds for $j=k$ and completes the induction.
\end{proof}

Now we turn our attention to right limits and normalized relative ratio asymptotics.  Our first task is to prove Theorem \ref{iffrlim}, which concerns a necessary and sufficient condition for the matrix $M_{\mu}$ to converge to a right limit through a subsequence $\mcn$.

\begin{proof}[Proof of Theorem \ref{iffrlim}:]
We use (\ref{cram1}) again to write
\begin{align}\label{cram2}
\frac{P_{n+j-1}(z;\mu)}{P_{n+j}(z;\mu)}&=\rho_{M_{\mu},n+j}(z)_{n+j,n+j}=\sum_{m=0}^{\infty}\frac{\left((\pi_{n+j}M_{\mu}\pi_{n+j})^m\right)_{n+j,n+j}}{z^{m+1}}.
\end{align}
Therefore we have convergence uniformly on $\{z:|z|\geq r\}$ for every $r>R_{\mu}$ if and only if each of the Laurent coefficients in (\ref{cram2}) converges.

Let us suppose that $\mcn\subseteq\bbN$ is a subsequence through which $M_{\mu}$ converges to a right limit.  Then $\mcn+j:=\{n+j:n\in\mcn,\,n+j>0\}$ is also a subsequence through which $M_{\mu}$ converges to a right limit.  Therefore, for every $j,k\in\bbZ$, the following limit exists:
\[
\lim_{{\nri}\atop{n\in\mcn}}(M_{\mu})_{n-j,n-j-k}.
\]
Lemma \ref{hprod} implies that the coefficient $\left((\pi_{n+j}M_{\mu}\pi_{n+j})^m\right)_{n+j,n+j}$ appearing in (\ref{cram2}) can be written as a finite sum of products of elements of this form.  Therefore, the Laurent coefficients in (\ref{cram2}) converge as $\nri$ through $\mcn$, and hence we have the desired ratio asymptotics for the monic orthogonal polynomials.  However, convergence to a right limit also implies convergence of the ratio $\kappa_{n+j-1}\kappa_{n+j}^{-1}$ to a ($j$-dependent) limit as $\nri$ through $\mcn$.  Therefore, we have shown that condition (I) of the theorem holds.  Notice that this half of the proof did not make use of the condition (\ref{kapo}).

Now let us suppose that there are functions $\{f_j\}_{j\in\bbZ}$ as in the statement of condition (I) of the theorem.  In this case, we know that $\kappa_{n+j-1}\kappa_{n+j}^{-1}$ converges to a ($j$-dependent) limit as $\nri$ through $\mcn$.  In other words, we know that the following limits exist:
\[
\lim_{{\nri}\atop{n\in\mcn}}M_{n+j+1,n+j},\qquad j\in\bbZ.
\]
Since we are assuming that each of these limits is non-zero, we know that each of the Laurent coefficients appearing in (\ref{cram2}) converges as $\nri$ through $\mcn$.  
If we examine the $m=1$ term in (\ref{cram2}), we also conclude that the following limits exist:
\[
\lim_{{\nri}\atop{n\in\mcn}}M_{n+j,n+j},\qquad j\in\bbZ.
\]
This will serve as the base case of our induction.  Suppose as our induction hypothesis that the following limits exist for all integers $k<q$:
\[
\lim_{{\nri}\atop{n\in\mcn}}M_{n+j-k,n+j},\qquad j\in\bbZ.
\]
We will show that the same limits exist for $k=q$, which will prove that $M_{\mu}$ converges to a right limit as $\nri$ through $\mcn$.

Consider then the coefficient of $z^{-q-2}$ in (\ref{cram2}).  Lemma \ref{hprod} implies that we can write
\begin{align}\label{longsum1}
\left((\pi_{n+j}M_{\mu}\pi_{n+j})^{q+1}\right)_{n+j,n+j}=h_{q,n+j}(\mu)+\beta_{n+j,q+1},
\end{align}
where the induction hypothesis implies that $\beta_{n+j,q+1}$ converges as $\nri$ through $\mcn$.  Since the left hand side of (\ref{longsum1}) also converges as $\nri$ through $\mcn$, we conclude that $h_{q,n+j}(\mu)$ converges as $\nri$ through $\mcn$.  To complete the proof of this half of the theorem, we need only recall that
\begin{align}\label{ratm}
\frac{\kappa_{n-2}(\mu)}{\kappa_{n-1}(\mu)}=\mmu_{n,n-1}.
\end{align}
Our assumptions tell us that each ratio $\{\kappa_{n+\ell-1}\kappa_{n+\ell}^{-1}\}_{\ell\in\bbZ}$ converges as $\nri$ through $\mcn$ and (\ref{kapo}) implies that the limit is not zero.  Therefore, it must be the case that $\lim_{\nri}\mmu_{n+j-q,n+j}$ exists for all $j\in\bbZ$ as $\nri$ through $\mcn$.  This completes the induction.
\end{proof}

As we remarked earlier, the proof of Theorem \ref{iffrlim} shows that the right limit of $M_{\mu}$ determines the functions $\{f_j\}_{j\in\bbZ}$ appearing in (\ref{fjdef}) and also shows that if (\ref{kapo}) is satisfied, then the right limit of $M_{\mu}$ is determined by the functions $\{f_j\}_{j\in\bbZ}$ appearing in (\ref{fjdef}) .  This observation will enable us to prove Corollary \ref{genrelrat}, which concerns a necessary and sufficient condition for normalized relative ratio asymptotics through subsequences.

\begin{proof}[Proof of Corollary \ref{genrelrat}:]
Suppose $\mu$ and $\nu$ exhibit normalized relative ratio asymptotics through $\mcn(\mu)$ and $\mcn(\nu)$.  Let $\mcn^*$ be a subsequence so that $M_{\mu}$ converges to a right limit $X$ as $\nri$ through $\{n_k\}_{k\in\mcn^*}$.  By Theorem \ref{iffrlim}, we know that for each $j\in\bbZ$, there is a function $f_j$ so that
\[
\lim_{{\kri}\atop{k\in\mcn^*}}\frac{p_{n_k+j-1}(z;\mu)}{p_{n_k+j}(z;\mu)}=f_j(z),\qquad |z|>R_{\mu}.
\]
From the definition of normalized relative ratio asymptotics, we see that these same equalities hold and with the same functions $\{f_j\}_{j\in\bbZ}$ if we replace $n_k$ by $m_k$ and $p_{n}(z;\mu)$ by $p_n(z;\nu)$.  Invoking Theorem \ref{iffrlim} again shows that $M_{\nu}$ converges to $X$ as $\nri$ through $\{m_k\}_{k\in\mcn^*}$.

For the converse, suppose condition (ii) in the statement of the corollary is true for the sequences $\mcn(\mu)$ and $\mcn(\nu)$.  Suppose for contradiction that there is a $z_0\in\{z:|z|>\max\{R_{\mu},R_{\nu}\}\}$, a $j\in\bbZ$, and a subsequence $\mcn^*$ so that
\[
\lim_{{\kri}\atop{k\in\mcn^*}}\left(\frac{p_{n_k+j-1}(z_0;\mu)}{p_{n_k+j}(z_0;\mu)}-\frac{p_{m_k+j-1}(z_0;\nu)}{p_{m_k+j}(z_0;\nu)}\right)=\mcz\neq0.
\]
By taking a subsequence of $\mcn^*$ if necessary, we may assume that both $M_{\mu}$ and $M_{\nu}$ converge to a right limit as $\nri$ through $\{n_k\}_{k\in\mcn^*}$ and $\{m_k\}_{k\in\mcn^*}$ respectively.  By assumption, these must be the same right limit, and hence the sequences $\{f_j\}_{j\in\bbZ}$ appearing in (\ref{fjdef}) must be the same.  However, this would imply that $\mcz=0$, which is a contradiction.

Notice that the proof of the implication (ii) implies (i) only relies upon the fact that condition (II) of Theorem \ref{iffrlim} implies condition (I) of Theorem \ref{iffrlim}.  Since that implication does not require the assumption (\ref{kapo}), this proves the last statement of the corollary.
\end{proof}

In the next section we will consider applications of the above results to random orthogonal polynomials on the unit circle and the real line.

\section{Random Orthogonal Polynomials}\label{random}

The results in Section \ref{intro} are most useful in settings where we have greater information about the structure of the Bergman Shift matrix.  In particular, if the measure $\mu$ is supported on the unit circle, then the Bergman Shift matrix is unitary and has entries that can be written explicitly in terms of the coefficients in the recursion relation satisfied by the monic orthogonal polynomials.  If $\mu$ is supported on a compact subset of the real line, then the Bergman Shift matrix is self-adjoint and banded and the entries are also expressed in terms of recursion coefficients.  In this section, we will explore some consequences of our results in cases when the measure is chosen randomly by choosing the recursion coefficients randomly.

\subsection{The Unit Circle Case}\label{randomopuc}

Suppose $\mu$ satisfies $\supp(\mu)\subseteq\partial\bbD$.  In this case, it is well-known that the entries of the Bergman shift matrix can be expressed in terms of the Verblunsky coefficients $\{\alpha_n\}_{n\geq0}$, where $P_n(0;\mu)=-\bar{\alpha}_{n-1}$ (see \cite[Section 4.1]{OPUC1}).  Verblunsky's Theorem (see \cite[Chapter 1]{OPUC1}) tells us that there is a one-to-one correspondence between infinite sequences in $\bbD$ and non-trivial probability measures on the unit circle.  Therefore, one can consider random measures on the unit circle by considering random sequences of Verblunsky coefficients (as in \cite{DS,KS,Stoiciu,Tep1}).  Our goal is to show that the orthonormal polynomials for a random measure almost surely exhibit a universal ratio asymptotic behavior.  The precise meaning of this statement depends on the random distribution from which we select our Verblunsky coefficients.

\vspace{2mm}

\noindent\textbf{Definition.}  If $\tau$ is a probability measure satisfying $\tau(\bbD)=1$, then the class $K(\tau)$ is the set of all probability measures on the unit circle whose Verblunsky coefficients are all contained in $\supp(\tau)\cap\bbD$.

\vspace{2mm}

Now we are ready to state our main result of this section.

\begin{theorem}\label{trandom}
Suppose $\tau$ is a probability measure satisfying $\tau(\bbD)=1$ and let $\mu$ be a probability measure on the unit circle chosen randomly by selecting the sequence of Verblunsky coefficients $\{\alpha_n\}_{n\geq0}$ as i.i.d. random variables with distribution $\tau$.  Then almost surely it is true that for every $\nu$ in the class $K(\tau)$, there exists a subsequence $\{m_n\}_{n=1}^{\infty}\subseteq\bbN$ such that
\[
\lim_{\nri}\left(\frac{p_{m_n-1}(z;\mu)}{p_{m_n}(z;\mu)}-\frac{p_{n-1}(z;\nu)}{p_{n}(z;\nu)}\right)=0,\qquad|z|>1.
\]
\end{theorem}

The proof of this result will be relatively straightforward given what we already know.  Indeed, since the set $\{p_{n-1}(z;\nu)/p_{n}(z;\nu)\}_{n\in\bbN}$ is a normal family, it suffices to show that whenever we have a subsequence through which the orthonormal polynomials for $\nu$ exhibit ratio asymptotics, there is a subsequence through which the orthonormal polynomials for $\mu$ exhibit the same ratio asymptotics.  We have seen that ratio asymptotic behavior is controlled by the right limit behavior of the Bergman Shift matrix, so it will suffice to show that the Bergman Shift matrix for the measure $\mu$ almost surely exhibits every possible kind of right limit behavior within the class $K(\tau)$.  This will follow from basic probability theory.

We begin with some useful lemmas.  We have already mentioned our definition of the right limit of a matrix, but for the right limit of a sequence, we use the same definition as in \cite{Rice}.  The form of the Bergman Shift matrix when $\supp(\mu)\subseteq\partial\bbD$ makes it clear that if the sequence of Verblunsky coefficients converges to a right limit through a subsequence $\mcn\subseteq\bbN$, then so does the corresponding Bergman Shift matrix and through the same subsequence $\mcn$.

\begin{lemma}\label{dense}
Let $\tau$ be a probability measure satisfying $\tau(\bbD)=1$.  There exists a sequence $\{z_n\}_{n\in\bbN}$ in $(\bbD\cap\supp(\tau))^{\bbN}$ that has all of $\supp(\tau)^{\bbZ}$ as right limits.
\end{lemma}

\begin{proof}
Let $S=\{s_1,s_2,\ldots\}$ be a countable sequence of points in $\bbD\cap\supp(\tau)$ that has all of $\supp(\tau)$ as limits of subsequences.  Let $S^{*}$ be the sequence obtained by setting the first element in the sequence equal to $s_1$, then following this with all permutations of $\{s_1,s_2\}$, then following this with all permutation of $\{s_1,s_2,s_3\}$, and so on.  For example, the beginning of the sequence $S^*$ is
\[
S^*=\{s_1,s_1,s_2,s_2,s_1,s_1,s_2,s_3,s_1,s_3,s_2,s_2,s_1,s_3,s_2,s_3,s_1,s_3,s_1,s_2,s_3,s_2,s_1,\ldots\}
\]
We claim that $S^*$ is our desired sequence.

To see this, let $\{a_j\}_{j\in\bbZ}$ be any element of $\supp(\tau)^{\bbZ}$ and let us write $S^*=\{u_1,u_2,\ldots\}$.  Since $S$ has all of $\supp(\tau)$ as limits of subsequences, we may for each $j\in\bbZ$ find a subsequence $\{m_{n,j}\}_{n\in\bbN}\subseteq\bbN$ such that
\[
\lim_{\nri}s_{m_{n,j}}=a_j.
\]
It is straightforward to see that we can thin these subsequences so that each natural number belongs to the subsequence $\{m_{n,j}\}_{n\in\bbN}$ for at most one integer $j$.
By construction, for every $k\in\bbN$, the $(2k+1)$-tuple $(s_{m_{k,-k}},s_{m_{k,-k+1}},\ldots,s_{m_{k,k}})$ appears infinitely often as a contiguous block in $S^*$.  So, we can build our desired subsequence $\mcn=\{n_1,n_2,\ldots\}$ by first setting $n_1$ equal to any $k_1$ where $u_{k_1}=s_{m_{1,0}}$.  For every $j\geq2$, we set $n_j$ equal to any $k_j>n_{j-1}$ such that
\[
u_{k_j-j+1}=s_{m_{j,-j+1}},\qquad\cdots,\qquad u_{k_j}=s_{m_{j,0}},\qquad\cdots,\qquad u_{k_j+j-1}=s_{m_{j,j-1}}.
\]
With this choice of $n_j$, we see that for any $\ell\in\bbZ$, we have
\[
\lim_{{\nri}\atop{n\in\mcn}}u_{n+\ell}=\lim_{j\rightarrow\infty}u_{n_j+\ell}=\lim_{j\rightarrow\infty}s_{m_{j,\ell}}=a_{\ell},
\]
which is what we wanted to show.
\end{proof}

\begin{lemma}\label{weakall}
Let $\tau$ be a probability measure satisfying $\tau(\bbD)=1$ and suppose $\{\alpha_n\}_{n\geq0}$ is a sequence of i.i.d. random variables with distribution $\tau$.  Then with probability one, all of $\supp(\tau)^{\bbZ}$ is a right limit of the sequence $\{\alpha_n\}_{n\geq0}$.
\end{lemma}  

\begin{proof}
Let $S^{*}=\{u_1,u_2,\ldots\}$ be the sequence whose existence is proven in Lemma \ref{dense}.  If $B(x,r)$ denotes the open ball with center $x$ and radius $r$ in the complex plane, then let $\{y_n\}_{n\in\bbN}$ be a sequence of positive real numbers converging to zero so that for every $j\in\bbN$
\[
\sum_{n=1}^{\infty}\prod_{k=1}^j\tau\left(B(u_k,y_n)\right)=\infty
\]
(it is straightforward to construct such a sequence).  Define the events $\{A_n^{(j)}\}_{n\in\bbN}$ by
\[
A_n^{(j)}=\left\{|\alpha_{n+k-1}-u_k|\leq y_n,\, k\in\{1,\ldots,j\}\right\},
\]
that is, $A_n^{(j)}$ is the event that $\{\alpha_n,\ldots,\alpha_{n+j-1}\}$ very closely resembles the beginning of the sequence $S^*$.  It is clear that for each $j\in\bbN$, the events $\{A_{3jn}^{(j)}\}_{n\in\bbN}$ are independent and
\[
\sum_{n=1}^{\infty}\bbP(A_{3jn}^{(j)})=\sum_{n=1}^{\infty}\prod_{k=1}^j\tau(B(u_k,y_n))=\infty.
\]
Therefore, the second Borel Cantelli Lemma (see \cite[Theorem 7.2.2]{Analysis1}) tells us that with probability $1$ we can find a subsequence $\mcn_j\subseteq\bbN$ such that
\[
\lim_{{\nri}\atop{n\in\mcn_j}}\alpha_{n+k-1}=u_k,\qquad k\in\{1,\ldots,j\}.
\]
It follows that we can construct a subsequence $\{n_k\}_{k=1}^{\infty}$ by choosing any $n_1\in\mcn_1$ and then choosing $n_j\in\mcn_j$ sufficiently large compared to $n_{j-1}$ for every $j\geq2$ so that
\begin{align}\label{ulim}
\lim_{j\rightarrow\infty}\alpha_{n_j+k-1}=u_k,\qquad k\in\bbN.
\end{align}
We conclude that with probability one, there is a subsequence $\{n_k\}_{k=1}^{\infty}$ so that (\ref{ulim}) holds.  However, it is straightforward to see that (\ref{ulim}) is sufficient to imply that every right limit of the sequence $S^*$ is also a right limit of the sequence $\{\alpha_n\}_{n\in\bbN}$.  By construction, the set of right limits of $S^*$ is all of $\supp(\tau)^{\bbZ}$, so the desired conclusion follows.
\end{proof}

\begin{proof}[Proof of Theorem \ref{trandom}:]
Suppose $\mu$ and $\nu$ are as in the statement of the theorem.  We begin with a simple observation: if $\mcn\subseteq\bbN$ is a subsequence such that $M_{\nu}$ converges to a right limit $X$ as $\nri$ through $\mcn$, then by Lemma \ref{weakall} there is almost surely a subsequence $\mcn'\subseteq\bbN$ such that $M_{\mu}$ converges to the same right limit $X$ as $\nri$ through $\mcn'$.  

Define
\[
\mcf(n,k,r):=\inf_{m>k}\left(\sup_{|z|>r}\left\{\left|\frac{p_{m-1}(z;\mu)}{p_{m}(z;\mu)}-\frac{p_{n-1}(z;\nu)}{p_{n}(z;\nu)}\right|\right\}\right).
\]
It suffices to show that for every $r>1$ and every subsequence $\{j_n\}_{n\in\bbN}\subseteq\bbN$ it holds that
\begin{align}\label{fshow}
\lim_{\nri}\mcf(n,j_n,r)=0.
\end{align}
Suppose for contradiction that there exists an $r_0>1$ and subsequences $\{k_n\}_{n\in\bbN},\{j_n\}_{n\in\bbN}\subseteq\bbN$ so that
\[
\lim_{\nri}\mcf(k_n,j_{k_n},r_0)=t>0.
\]
Take a subsequence $\{k_n'\}_{n\in\bbN}$ of $\{k_n\}_{n\in\bbN}$ so that $M_{\nu}$ converges to a right limit $X_0$ as $\nri$ through $\{k_n'\}_{n\in\bbN}$.  By our earlier observation, there is almost surely a subsequence $\{h_n\}_{n\in\bbN}\subseteq\bbN$ such that $M_{\mu}$ converges to the right limit $X_0$ as $\nri$ through $\{h_n\}_{n\in\bbN}$ and it is trivial to see that we may refine this subsequence so that $h_n>j_{k_n}$ for all $n\in\bbN$.  Therefore, by Corollary \ref{genrelrat} it is true that for all sufficiently large $n$ it holds that
\[
\mcf(k_n,j_{k_n},r_0)\leq\sup_{|z|>r_0}\left|\frac{p_{h_n-1}(z;\mu)}{p_{h_n}(z;\mu)}-\frac{p_{k_n-1}(z;\nu)}{p_{k_n}(z;\nu)}\right|<\frac{t}{2},
\]
which gives us the desired contradiction.
\end{proof}

\subsection{The Real Line Case}\label{randomoprl}

Suppose $\mu$ satisfies $\supp(\mu)\subseteq\bbR$.  In this case, it is well-known that the entries of the Bergman shift matrix can be expressed in terms of the Jacobi parameters $\{a_n,b_n\}_{n\geq1}$, where $a_n>0$ and $b_n\in\bbR$.  Indeed, the diagonal entries of the Bergman Shift matrix are the sequence $\{b_n\}_{n\in\bbN}$ and the off-diagonal elements are the sequence $\{a_n\}_{n\in\bbN}$.  Favard's Theorem (see \cite[Chapter 1]{OPUC1}) tells us that there is a one-to-one correspondence between bounded sequences of Jacobi parameters and non-trivial and compactly supported probability measures on the real line.  Therefore, one can consider random measures on the real line by considering random sequences of Jacobi parameters.  Our goal is to show that the orthonormal polynomials for a random measure almost surely exhibit a universal ratio asymptotic behavior.  The precise meaning of this statement depends on the random distribution from which we select our Jacobi parameters.

\vspace{2mm}

\noindent\textbf{Definition.}  If $\tau_1$ is a compactly supported probability measure on $(0,\infty)$ and $\tau_2$ is a compactly supported probability measure on $\bbR$, then the class $K(\tau_1,\tau_2)$ is the set of all probability measures on the real line whose Jacobi parameters $\{a_n,b_n\}_{n\in\bbN}$ satisfy
\[
a_n\in\supp(\tau_1),\qquad b_n\in\supp(\tau_2),\qquad n\in\bbN.
\]
The class $K^*(\tau_1,\tau_2)$ is the set of all bi-infinite tri-diagonal symmetric matrices having diagonal entries in $\supp(\tau_2)$ and off-diagonal entries in $\supp(\tau_1)$.

\vspace{2mm}

Now we are ready to state the analog of Theorem \ref{trandom} for measures on the real line.

\begin{theorem}\label{ttrandom}
Suppose  $\tau_1$ is a compactly supported probability measure on $(0,\infty)$ and $\tau_2$ is a compactly supported probability measure on $\bbR$.  Let $\mu$ be a probability measure on the real line chosen randomly by selecting the sequence of off-diagonal Jacobi parameters $\{a_n\}_{n\in\bbN}$ as i.i.d. random variables with distribution $\tau_1$ and diagonal Jacobi parameters $\{b_n\}_{n\in\bbN}$ as i.i.d. random variables with distribution $\tau_2$.  Then almost surely it is true that for every $\nu$ in the class $K(\tau_1,\tau_2)$, there exists a subsequence $\{m_n\}_{n=1}^{\infty}\subseteq\bbN$ such that
\[
\lim_{\nri}\left(\frac{p_{m_n-1}(z;\mu)}{p_{m_n}(z;\mu)}-\frac{p_{n-1}(z;\nu)}{p_{n}(z;\nu)}\right)=0,\qquad|z|>\max\{R_{\mu},R_{\nu}\}.
\]
\end{theorem}

The proof of Theorem \ref{ttrandom} is very similar to the proof of Theorem \ref{trandom}.  The only substantive modification is in the analog of Lemma \ref{dense}, which we now provide.

\begin{lemma}\label{realall}
Suppose  $\tau_1$ is a compactly supported probability measure on $(0,\infty)$ and $\tau_2$ is a compactly supported probability measure on $\bbR$.  There exists a Jacobi matrix $J^*$ in the class $K(\tau_1,\tau_2)$ that has all of $K^*(\tau_1,\tau_2)$ as right limits.
\end{lemma}

\begin{proof}
Let $S=\{s_1,s_2,\ldots\}$ be a countable set in $\supp(\tau_1)$ that has all of $\supp(\tau_1)$ as limits of subsequences and let $T=\{t_1,t_2,\ldots\}$ be a countable set in $\supp(\tau_2)$ that has all of $\supp(\tau_2)$ as limits of subsequences.  Let $S^*$ be the sequence obtained by setting the first element in the sequence equal to $s_1$, then following this with $2$ consecutive appearances of each permutation of $\{s_1,s_2\}$.  The next elements of $S^*$ will be $6$ consecutive appearances of each permutation of $\{s_1,s_2,s_3\}$.  This pattern is repeated by including $n!$ appearances of each permutation of $\{s_1,\ldots,s_n\}$.  For example, the beginning of the sequence $S^*$ is
\[
S^*=\{s_1,s_1,s_2,s_1,s_2,s_2,s_1,s_2,s_1,s_1,s_2,s_3,s_1,s_2,s_3,s_1,s_2,s_3,s_1,s_2,s_3,s_1,s_2,s_3,s_1,s_2,s_3,\ldots\}.
\]
Let $T^*$ be the sequence obtained by setting the first element in the sequence equal to $t_1$, then following this with every possible permutation of $\{t_1,t_2\}$ repeated $2$ times, then following this with every possible permutation of $\{t_1,t_2,t_3\}$ repeated $6$ times, and so on.  At each step of the construction, we will add on every possible permutation of $\{t_1,\ldots,t_n\}$ repeated $n!$ times.  For example, the beginning of the sequence $T^*$ is
\[
T^*=\{t_1,t_1,t_2,t_2,t_1,t_1,t_2,t_2,t_1,t_1,t_2,t_3,t_1,t_3,t_2,t_2,t_1,t_3,t_2,t_3,t_1,t_3,t_1,t_2,t_3,t_2,t_1\ldots\}.
\]
We claim that we may set $J^*$ equal to the matrix with the sequence $T^*$  on the diagonal and $S^*$ on the off-diagonals.

To see this, let $\{a_j\}_{j\in\bbZ}$ be any element of $\supp(\tau_1)^{\bbZ}$ and let $\{b_j\}_{j\in\bbZ}$ be any element of $\supp(\tau_2)^{\bbZ}$.  Let us also write $S^*=\{u_n\}_{n\in\bbN}$ and $T^*=\{v_n\}_{n\in\bbN}$.  By the properties of $S$ and $T$, we may for each $j,k\in\bbZ$ find subsequences $\{m_{n,j}\}_{n\in\bbN},\{q_{n,k}\}_{n\in\bbN}\subseteq\bbN$ such that
\[
\lim_{\nri}s_{m_{n,j}}=a_j,\qquad\qquad\lim_{\nri}t_{q_{n,k}}=b_k.
\]
It is straightforward to see that we can thin these subsequences so that each natural number belongs to the subsequence $\{m_{n,j}\}_{n\in\bbN}$ for at most one integer $j$ and to the subsequence $\{q_{n,k}\}_{n\in\bbN}$ for at most one integer $k$.

By construction, for every $k\in\bbN$, the $(2k+1)$-tuple $(s_{m_{k,-k}},s_{m_{k,-k+1}},\ldots,s_{m_{k,k}})$ appears infinitely often as a contiguous block in $S^*$ and at the same position in $S^*$ as the block $(t_{q_{k,-k}},t_{q_{k,-k+1}},\ldots,t_{q_{k,k}})$ in $T^*$.  So, we can build our desired subsequence $\mcn=\{n_1,n_2,\ldots\}$ by first setting $n_1$ equal to any $\ell_1$ where $u_{\ell_1}=s_{m_{1,0}}$ and $v_{\ell_1}=t_{q_{1,0}}$.  For every $j\geq2$, we set $n_j$ equal to any $\ell_j>n_{j-1}$ such that
\begin{align*}
u_{\ell_j-j+1}&=s_{m_{j,-j+1}},\qquad\cdots,\qquad u_{\ell_j}=s_{m_{j,0}},\qquad\cdots,\qquad u_{\ell_j+j-1}=s_{m_{j,j-1}}\\
v_{\ell_j-j+1}&=t_{q_{j,-j+1}},\qquad\cdots,\qquad v_{\ell_j}=t_{q_{j,0}},\qquad\cdots,\qquad v_{\ell_j+j-1}=t_{q_{j,j-1}}.
\end{align*}
With this choice of $n_j$, we see that for any $r\in\bbZ$, we have
\begin{align*}
\lim_{{\nri}\atop{n\in\mcn}}u_{n+r}&=\lim_{\jri}u_{n_j+r}=\lim_{\jri}s_{m_{j,r}}=a_r\\
\lim_{{\nri}\atop{n\in\mcn}}v_{n+r}&=\lim_{\jri}v_{n_j+r}=\lim_{\jri}t_{q_{j,r}}=b_r,
\end{align*}
as desired.
\end{proof}

With Lemma \ref{realall} in hand, it is now a simple matter to adapt the proof of Theorem \ref{trandom} to complete the proof of Theorem \ref{ttrandom}, so we omit the details.

\vspace{2mm}

In the next section, we will consider some additional examples that highlight some important applications of the results from Section \ref{intro}.

\section{Further Examples}\label{examp}

\subsection{Example: Alexandrov Measures on the Unit Circle}\label{alex}

Suppose $\mu$ is a probability measure supported on $\partial\bbD$ with Verblunsky coefficients $\{\alpha_n\}_{n=0}^{\infty}$.  For every $\lambda\in\partial\bbD$, one can consider the measure $\mu_{\lambda}$, whose Verblunsky coefficients are related to those of $\mu$ by
\[
\alpha_n(\mu_{\lambda})=\lambda\alpha_n(\mu),\qquad\qquad\qquad n\in\{0,1,2,3,\ldots\}.
\]
The family of measures $\{\mu_{\lambda}\}_{\lambda\in\partial\bbD}$ is called the family of \textit{Alexandrov Measures} for the measure $\mu$ (see \cite{OPUC1}).

If $\lambda_1$ and $\lambda_2$ are distinct complex numbers in the unit circle, then the relationship between $\mu_{\lambda_1}$ and $\mu_{\lambda_2}$ is highly non-trivial.  Indeed, it is well-known that these two measures need not be mutually absolutely continuous (see \cite[Section 10.3]{OPUC2}).  However, the structure of the Bergman Shift matrix reveals that $M_{\mu_{\lambda_1}}$ and $M_{\mu_{\lambda_2}}$ differ only in the first row, and thus the unique right limit of their difference is zero.  Indeed, the formulas in \cite[Section 4.1]{OPUC1} or \cite[Equation 28]{GNVA} tell us that
\[
(M_{\mu})_{ij}=\begin{cases}
-\alpha_{i-2}\bar{\alpha}_{j-1}\prod_{k=i-1}^{j-2}\sqrt{1-|\alpha_k|^2} & \quad i\leq j,\\
\sqrt{1-|\alpha_{j-1}|^2} & \quad i=j+1,
\end{cases}
\]
where we set $\alpha_{-1}=-1$ and all unspecified entries equal to zero.  It follows at once from Corollary \ref{genrelrat} that the orthonormal polynomials for $\mu_{\lambda_1}$ and $\mu_{\lambda_2}$ exhibit normalized relative ratio asymptotics.  We remark that this fact could also be deduced from \cite[Theorem 1.5]{GolZlat}.

\subsection{Example: Coefficient Stripping}

If $\{\alpha_0,\alpha_1,\ldots\}$ is a sequence of Verblunsky coefficients for a measure $\mu$ supported on $\partial\bbD$, then we can associate to it the measure $\mu^1$, which has Verblunsky coefficient sequence $\{\alpha_1,\alpha_2,\ldots\}$ and is called the \textit{once stripped} measure (see \cite[Section 3.4]{OPUC1}).  Similarly, for every $k\in\bbN$ the measure $\mu^k$ is defined as the measure corresponding to the Verblunsky coefficient sequence $\{\alpha_k,\alpha_{k+1},\ldots\}$.  One often refers to the polynomials corresponding to the measure $\mu^k$ as the $k^{th}$ associated polynomials (see \cite{VA91}).  It is easy to see that if $M_{\mu}$ approaches a right limit as $\nri$ through $\mcn\subseteq\bbN$, then $M_{\mu^k}$ approaches that same right limit as $\nri$ through $\mcn-k$.  It follows from Corollary \ref{genrelrat} that
\[
\lim_{\nri}\left[\frac{p_{n-k-1}(z;\mu^k)}{p_{n-k}(z;\mu^k)}-\frac{p_{n-1}(z;\mu)}{p_{n}(z;\mu)}\right]=0,\qquad\qquad|z|>1.
\]
A similar result holds for stripped orthogonal polynomials on the real line.

\subsection{Example: Measures on the Real Line}

In this example, we will revisit a result proven by Nevai and Van Assche in \cite{NVA92}.  Let us consider the case in which the measures of orthogonality $\mu$ and $\nu$ are each supported on (perhaps different) compact subsets of the real line.  In this case, the Bergman Shift matrix $M_{\mu}$ is symmetric and is zero away from its three main diagonals (and similarly for $M_{\nu}$).  If we label the diagonal elements of $M_{\mu}$ as $\{b_n(\mu)\}_{n\in\bbN}$ and the off diagonal elements as $\{a_n(\mu)\}_{n\in\bbN}$, then the hypothesis (\ref{kapo}) is equivalent to the condition $\inf_{n\in\bbN} a_n(\mu)>0$.  Corollary \ref{genrelrat} then implies that if $\inf_{n\in\bbN}a_n(\mu)a_n(\nu)>0$, then $\mu$ and $\nu$ exhibit normalized relative ratio asymptotics if and only if $M_{\mu}-M_{\nu}$ is compact, which in turn implies that the essential support of $\mu$ and $\nu$ is the same (this is by Weyl's Theorem; see \cite[Theorem S.13]{RS1}).

\vspace{2mm}

\subsection{Example: Degenerate Cases}

If the hypothesis (\ref{kapo}) is removed, then conclusion of Corollary \ref{genrelrat} fails because it is possible that both $\mu$ and $\nu$ exhibit normalized ratio asymptotics with limit function $0$, but the matrix $M_{\mu}-M_{\nu}$ has a non-zero right limit.  To make this more concrete, we appeal to measures on the unit circle.

Define the measures $\mu$ and $\nu$ through their Verblunsky coefficients by
\[
\alpha_n(\mu)=1-\frac{1}{n+2},\qquad\alpha_n(\nu)=\left(1-\frac{1}{n+2}\right)e^{in},\qquad n\geq0.
\]
A similar pair of measures provided an illustrative example in \cite{RatioHessI}.  It is well-known that $\kappa_n(\mu)\kappa_{n+1}(\mu)^{-1}=\sqrt{1-|\alpha_n(\mu)|^2}$ (see \cite[Equation 1.5.22]{OPUC1}) and similarly for $\nu$.  Since $\{P_{n-1}(z;\mu)P_{n}(z;\mu)^{-1}\}_{n\in\bbN}$ is a normal family on $\{z:|z|>1\}$, we have
\[
\lim_{\nri}\frac{p_{n-1}(z;\mu)}{p_{n}(z;\mu)}=0=\lim_{\nri}\frac{p_{n-1}(z;\nu)}{p_{n}(z;\nu)},\qquad|z|>1,
\]
so $\mu$ and $\nu$ exhibit normalized relative ratio asymptotics.  However, by appealing to the formulas in Example \ref{alex}, we see that when $n>1$ we have
\[
(M_{\mu})_{n,n}=-\left(1-\frac{1}{n+1}\right)\left(1-\frac{1}{n}\right),\qquad (M_{\nu})_{n,n}=-\left(1-\frac{1}{n+1}\right)\left(1-\frac{1}{n}\right)e^{-i}.
\]
From this it follows easily that $0$ is not the unique right limit of $M_{\mu}-M_{\nu}$, and we see that the condition (\ref{kapo}) is necessary in the statement of Corollary \ref{genrelrat}.  Furthermore, the measure $\nu$ shows that (\ref{kapo}) is an essential assumption in Theorem \ref{iffrlim}.

\end{document}